\documentclass{amsart}
\usepackage{amsfonts}
\usepackage{comment}
\usepackage{amsmath}
\usepackage{amsthm}
\usepackage{amssymb}
\usepackage{enumitem}
\usepackage{tkz-fct}
\usepackage{tikz}
\usepackage{tikz-cd}

\newtheorem{theorem}{Theorem}[section]
\newtheorem{lemma}[theorem]{Lemma}
\newtheorem{cor}[theorem]{Corollary}
\newtheorem{remark}[theorem]{Remark}
\newtheorem{proposition}[theorem]{Proposition}
\newtheorem{example}[theorem]{Example}
\thanks{Research of K.~E.~Hare was supported by NSERC Grant 2016-03719}
\thanks{Research of K.~G.~Hare was supported by NSERC Grant 2014-03154}
\thanks{Research of B.~Morris was supported by Fields Undergraduate Summer Research Program}
\thanks{Research of W.~Shen was supported by the NSERC USRA program and Grants 2014-03154 and 2016-03719}
\numberwithin{equation}{section}

\begin{document}
\title{The Entropy of Cantor--like measures}
\author{Kathryn E. Hare}
\address{Dept. of Pure Mathematics, University of Waterloo, Waterloo, Ont.,
N2L 3G1, Canada}
\email{kehare@uwaterloo.ca}
\author{Kevin G. Hare}
\address{Dept. of Pure Mathematics, University of Waterloo, Waterloo, Ont.,
N2L 3G1, Canada}
\email{kghare@uwaterloo.ca}
\author{Brian P. M. Morris}
\address{Dept. of Mathematics, Stanford University, Stanford, CA, 94305, USA}
\email{bpmm@stanford.edu}
\author{Wanchun Shen}
\address{Dept. of Pure Mathematics, University of Waterloo, Waterloo, Ont.,
N2L 3G1, Canada}
\email{w35shen@edu.uwaterloo.ca}
\maketitle

\begin{abstract}
By a Cantor-like measure we mean the unique self-similar probability measure 
$\mu $ satisfying $\mu =\sum_{i=0}^{m-1}p_{i}\mu \circ S_{i}^{-1}$ where $%
S_{i}(x)=\frac{x}{d}+\frac{i}{d}\cdot \frac{d-1}{m-1}$ for integers $2\leq
d<m\le 2d-1$ and probabilities $p_{i}>0$, $\sum p_{i}=1$. In the uniform
case ($p_{i}=1/m$ for all $i$) we show how one can compute the entropy and
Hausdorff dimension to arbitrary precision. In the non-uniform case we find
bounds on the entropy.
\end{abstract}

\section{Introduction}

By a self-similar measure we mean the unique probability measure 
\begin{equation*}
\mu =\sum_{i=0}^{m-1}p_{i}\mu \circ S_{i}^{-1},
\end{equation*}%
where $S_{i}$ are linear contractions on $\mathbb{R}$ and $p_{i}>0$ are
probabilities with $\sum_{i=0}^{m-1}p_{i}=1$. We restrict our attention to
those measures whose support is $[0,1]$. It is known that these self-similar
measures are either purely singular or absolutely continuous with respect to
Lebesgue measure \cite{JW}, but it is often difficult to determine which is
the case for a particular example.

An interesting class of examples are the Bernoulli convolutions where $m=2$, 
$S_{0}(x)=x/\varrho $ and $S_{1}=x/\varrho +1-1/\varrho $ for some $\varrho \in (1,2)$.
These have been extensively studied since the 1930's when Erd{\"{o}}s \cite%
{Erdos39} showed that if $\varrho \in (1,2)$ was a Pisot number, then the
Bernoulli convolution was purely singular and later, in \cite{Erdos40}, that
the Bernoulli convolutions were absolutely continuous for almost all $\varrho
\in (1,2)$. For more on the history of these classical problems see \cite%
{Sixty, Va}.

In \cite{Ga}, Garsia showed that the notion of entropy was useful for
studying the dimensional properties of Bernoulli convolutions. Subsequently,
the Garsia entropy was computed for various Bernoulli convolutions, first
for $\varrho =(1+\sqrt{5})/2$, the golden ratio (a simple Pisot number) in \cite%
{az}, then for all simple Pisot numbers in \cite{bo, gkt}, and, finally, for
all algebraic integers in \cite{AFKP, BV}. Edson, in \cite{Edson},
generalized these results in a different direction, considering the
contraction factor $1/\varrho $ where $\varrho $ is the root of $x^{2}-ax-b$ with $%
a\geq b$ and $a$ equally spaced linear contractions.

This paper focuses on a different generalization, to the case where $\varrho =d$
is an integer greater than or equal to $2$ and $m$ equally spaced
contractions $S_{j}$ of the form 
\begin{equation}
S_{j}(x)=\frac{x}{d}+\frac{j}{d}\cdot \frac{d-1}{m-1}\hspace{1cm}j=0,1,\dots
,m-1  \label{Sj}
\end{equation}%
where $2\leq d<m\leq 2d-1$. If, for example, $d=3$, $m=4$ and the
probabilities satisfy $p_{0}=p_{3}=1/8$, $p_{1}=p_{2}=3/8$, then the
associated self-similar measure is the $3$-fold convolution product of the
classical middle-third Cantor measure. We call these self-similar
Cantor-like measures,\textit{\ }$(m,d)$\textit{-measures}, and refer to them
as \textit{uniform} if all $p_{i}$ are equal. The dimensional properties of
these measures are also of much interest; see, for example, \cite{BHM, HHN, LW, Sh}.

We use combinatorial techniques to find an (explicit) analytic function $T$
with the property that the Garsia entropy of the uniform $(m,d)$-measure is
given by $T(1)/\log _{2}d$ when $2\leq d<m\leq 2d-1$. This is done in
Section 3 where we first illustrate the method with the simple case $d=2,$ $%
m=3$, and then handle the general uniform $(m,d)$-measure. Bounds are found
for the Garsia entropy of the non-uniform $(m,d)$-measures, a more
complicated problem, in Section 4. In Section 5, we use precise information
about the function $T$ to give numerically significant estimates for the
entropy in the uniform case when $2\leq d\leq 10$ and give ranges for the
value of the entropy for some non-uniform examples. We begin, in Section \ref%
{sec:discussion}, with the definition of the Garsia entropy and a discussion
of some of the combinatorial ideas we use in the proofs.

As with sets, there is a notion of the Hausdorff dimension of a probability
measure $\mu $ defined as 
\begin{equation*}
\dim _{H}\mu =\inf \{\dim _{H}E:\mu (E)>0\}.
\end{equation*}%
If $\mu $ is a measure on $\mathbb{R}$ and $\dim _{H}\mu <1,$ then $\mu $ is
singular. For self-similar measures arising from a set of contractions that
satisfy the open set condition there is a simple formula for computing $\dim
_{H}\mu $ (c.f. \cite{Fa}). But neither the Bernoulli convolutions nor the
Cantor--like measures satisfy this separation property and their Hausdorff
dimensions can be difficult to compute. It is a deep result of Hochman \cite%
{Hochman2014} (see \cite{BV} for details) that the Hausdorff dimension of
measures on $\mathbb{R}$ satisfying a suitable separation condition (which
includes Bernoulli convolutions with contraction factor an algebraic number
and the $(m,d)$-measures) is the minimum of $1$ and the Garsia entropy of
the measure, thus our results also give new estimates on the Hausdorff
dimensions of these measures.

The Hausdorff dimension of a self-similar measure can also be found from its 
$L^{q}$ spectrum; see \cite{Ng} for details. Using this approach, infinite
series representations have been found in \cite{LN} for the Hausdorff
dimension of the $(d,d+1)$-measures and in \cite{Feng} for Bernoulli
convolutions with contraction factor the inverse of a of simple Pisot
number. These involve matrix products, hence are less computationally
efficient.\texttt{\ }Numerical values (to four digits) were given in \cite%
{LN} for the special case of the $3$-fold convolution of the classic Cantor
measure.

\section{A combinatorial approach to the Garsia entropy}

\label{sec:discussion}

\subsection{The $(m,d)$-graph and entropy of the $(m,d)$-measure}

We will take a combinatorial approach to studying the Garsia entropy of the
Cantor-like $(m,d)$-measures $\mu =\sum_{j=0}^{m-1}p_{j}\mu \circ
S_{j}^{-1}, $ with $S_{j}$ as in (\ref{Sj}) and integers $d,m$ satisfying $%
2\leq d<m\leq 2d-1$. \ 

For this, we will need to introduce further notation. Given $\sigma \in
\{0,1,\dots ,m-1\}^{n}$, say $\sigma =(\sigma _{1},\sigma _{2},\dots ,\sigma
_{n}),$ we set $S_{\sigma }=S_{\sigma _{1}}\circ \dots \circ S_{\sigma _{n}}$
and call $\sigma $ a word of length $\left\vert \sigma \right\vert =n$. We
write $p_{\sigma }$ for the product $p_{\sigma _{1}}\cdot \cdot \cdot
p_{\sigma _{n}}$.

It is possible for $S_{\sigma }=S_{\tau }$ with $\left\vert \sigma
\right\vert =\left\vert \tau \right\vert $, but $\sigma \neq \tau $; the
Garsia entropy takes into account how often these overlaps occur and the
associated probabilities. To compute this, we create a graph where there is
a single root, which we can think of as $S_{\emptyset }(0)$. The nodes at
level $n\geq 1$ are the distinct images $S_{\sigma }(0)$ for $|\sigma |=n$
and a node $S_{\sigma }(0)$ at level $n$ is connected to all the nodes of
the form $S_{\sigma i}(0)$, $i=0,1,\dots,m-1$ at level $n+1$. We call this
the $(m,d)$\emph{-graph}. See Figure \ref{3-2-whole-graph} for an example.

Denote by $g_{n}$ the set of nodes at level $n\geq 0$ in the graph. For $%
z\in g_{n}$, we will denote by $[z]_{n}$ the set of all $\sigma $ with $%
\left\vert \sigma \right\vert =n$ such that $z=$ $S_{\sigma }(0)$. We assign
weight $w_{z}$ to the node $z$, where 
\begin{equation*}
w_{z}=\sum_{\sigma \in \lbrack z]_{n}}p_{\sigma }
\end{equation*}%
and we let $W_{n}$ denote the set of all weights $w_{z}$ associated to some $%
z\in g_{n}$.

The entropy of the $n$\textsuperscript{th} level of the weighted $(m,d)$%
-graph associated with the $(m,d)$-measure $\mu $ is defined as%
\begin{equation}
h_{\mu }(n)=-\sum_{w_{z}\in W_{n}}w_{z}\log _{2}w_{z}  \label{hmu}
\end{equation}%
and the \emph{Garsia entropy }$\mathfrak{H}_{\mu }$ (hereafter called the%
\emph{\ entropy}) of $\mu $ is given by%
\begin{equation}
\mathfrak{H}_{\mu }=\lim_{n\rightarrow \infty }\frac{h_{\mu }(n)}{n\log _{2}d%
}.  \label{entropymu}
\end{equation}

Another way to describe this calculation is as follows. Put%
\begin{equation*}
\mu _{n}=\bigotimes {}_{k=1}^{n}\left( \sum_{j=1}^{n}p_{j}\delta _{\frac{%
j(d-1)}{m-1}d^{-k}}\right)
\end{equation*}%
where $\bigotimes $ denotes the convolution product. These discrete measures
converge weak $\ast $ to $\mu $ and supp$\mu _{n}=\{S_{\sigma
}(0):\left\vert \sigma \right\vert =n\}$. Denote by $D_{n}$ the partition of 
$[0,1]$ into $(m-1)d^{n}$ equally spaced points. Each subinterval $[a,b)$ of 
$D_{n}$ can be identified with a unique node at level $n$, namely the node $%
z $ such that $S_{\sigma }(0)$ belongs to the subinterval for $\sigma \in
\lbrack z]_{n}$. The weight, $w_{z}$, equals $\mu _{n}([a,b))$.

With this notation, we have 
\begin{equation*}
h_{\mu }(n)=-\sum_{\Delta \in D_{n}}\mu _{n}(\Delta )\log _{2}(\mu
_{n}(\Delta )).
\end{equation*}%
Thus one can see that $h_{\mu }(n)=H(\mu _{n},D_{n})$ (where $D_{n}$ is the
partition we have described, rather than the partition into $2^{n}$ equally
spaced points), in the notation of \cite{Hochman2014}.

In the case of the uniform $(m,d)$-measure, the measure $\mu ,$ and hence
also $h_{\mu }$ and $\mathfrak{H}_{\mu },$ depend only on $m$ and $d$, and
we will write $h_{m,d}$ and $\mathfrak{H}_{m,d}$. The entropy calculation
can be simplified in this case: We will let $\mathrm{freq}(z)$ denote the
number of paths from the root node to $z\in g_{n}$. As $p_{\sigma }=m^{-n}$
for all $\sigma $ of length $n$, $w_{z}=m^{-n}\mathrm{freq}(z)$. Thus 
\begin{equation}
h_{m,d}(n)=-\sum_{z\in g_{n}}(m^{-n}\mathrm{freq}(z))\log _{2}(m^{-n}\mathrm{%
freq}(z)).  \label{h}
\end{equation}

If we let $f_{m,d}(n,k)$ denote the number of nodes in level $n$ with
frequency $k,$ then we have 
\begin{eqnarray*}
h_{m,d}(n)&=&\sum_{k=1}^{\infty }\sum_{\substack{ z\in g_{n},  \\ \mathrm{%
freq}(z)=k}} m^{-n}k\log _{2}(m^{-n}k) \\
&=&-m^{-n}\sum_{k=1}^{\infty }f_{m,d}(n,k)k(-n\log _{2}m+\log _{2}k).
\end{eqnarray*}%
Since the total number of nodes at level $n$ (counted by frequency) is $%
m^{n},$ this reduces to%
\begin{equation}
h_{m,d}(n)=n\log _{2}m-m^{-n}\sum_{k=1}^{\infty }f_{m,d}(n,k)k\log _{2}k.
\label{hsimp}
\end{equation}

When $\mu ,m$ or $d$ are clear, we may suppress them in the notation.

\subsection{Generating functions associated with the $(m,d)$ graph}

For studying the entropy of the uniform $(m,d)$-measure it is helpful to
introduce generating functions associated with the $(m,d)$-graph: Denote by 
\begin{equation*}
H(x)=H_{m,d}(x)=\sum_{n=0}^{\infty }h_{m,d}(n)x^{n}
\end{equation*}%
the generating function for the entropies of the levels of the $(m,d)$-graph 
and denote by $F_{k}$ the generating function (for the number of nodes of
frequency $k$ at each level) of the $(m,d)$-graph, 
\begin{equation*}
F_{k}(x)=\sum_{n=0}^{\infty }f(n,k)x^{n},
\end{equation*}%
and the related function 
\begin{equation*}
\mathcal{F}(x,s)=\sum_{k=2}^{\infty }k^{s}F_{k}(x).
\end{equation*}%
Since we assume $m\leq 2d-1$, the largest frequency at level $n$ is at most
twice the largest frequency at level $n-1$ and thus $f(n,k)=0$ if $k>2^{n}$.
Further, $f(n,k)\leq m^{n}$, hence $F_{k}(x)=\sum_{n\geq \log
_{2}k}f(n,k)x^{n}$ and 
\begin{equation*}
\left\vert F_{k}(x)\right\vert \leq \sum_{n\geq \log _{2}k}m^{n}\left\vert
x\right\vert ^{n}\leq c\varepsilon ^{\log _{2}k}\text{ if }\left\vert
x\right\vert \leq \varepsilon /m.
\end{equation*}%
It follows from these bounds that for $x$ small enough, $\frac{\partial }{%
\partial s}\mathcal{F}(x,s)|_{s=1}$ can be obtained by differentiating the
series term-by-term.

\subsection{Euclidean tree}

The $(m,d)$-graph is closely connected to the Euclidean tree, as we will
explain in Sections \ref{sec:2,1} and \ref{sec:d,r}, and will be helpful in
studying the entropy of the uniform $(m,d)$-measure. Here we describe the
construction of the Euclidean tree.

Start with two nodes connected by an edge, a root node with label $\{1,1\}$
at level $n=0$ and a node with label $\{2,1\}$ at level $n=1$. For each node 
$\{a,b\}$ in level $n\geq 1$, add two children with labels $\{a,a+b\}$ and $%
\{a+b,b\}$ in level $n+1$. This graph is the Euclidean tree and is
illustrated in Figure \ref{3-2-subgraph-with-euc-tree}.

Notice that all labels in the Euclidean tree are coprime pairs and that a
path from some node $\{a,b\}$ in the Euclidean tree to the root records the
steps involved in executing the simple Euclidean algorithm (the Euclidean
algorithm, but with repeated subtraction replacing division) on the pair $%
\{a,b\}$. Define $e(k,i)$ to be the number of steps it takes to reduce the
pair $\{k,i\}$ to their GCD via the simple Euclidean algorithm. For every
coprime pair $\{k,i\}$, $e(k,i)=n$ if and only if the pair $\{k,i\}$ is
found on level $n$ of the Euclidean tree. We refer the reader to \cite{az}
for further description and the history of the Euclidean tree.

Let $a(n,k)$ be the number of times that the integer $k$ occurs as the
larger value of a label at level $n$ of the Euclidean tree and let 
\begin{equation*}
A_{k}(x)=\sum_{n=0}^{\infty }a(n,k)x^{n}=\sum_{\substack{ 1\leq i\leq k  \\ %
\gcd (i,k)=1}}x^{e(k,i)}
\end{equation*}%
be the generating function (for occurrences of $k$ in level $n)$ of the
Euclidean tree. (In our notation, $A_{k}(x)$ is the function $\hat{\alpha}%
(x) $ in \cite{az}.) Each occurrence of $k\geq 2$ as a label (larger \textit{%
or} smaller) in the Euclidean tree can be traced up the Euclidean tree to an
occurrence as the larger label, and each time $k$ appears as a larger label
there is a single line of descendants in which it appears as the smaller
label. For instance, the 2 in the label $\{5,2\}$ on level $n=3$ of the
Euclidean tree can be traced back up to the label $\{1,2\}$ on level 1.
Thus, the family of generating functions for larger \textit{and} smaller
labels in the Euclidean tree is $(1+x+x^{2}+\ldots )A_{k}(x)=\frac{1}{1-x}%
A_{k}(x)$ for $k\geq 2$.

We define%
\begin{equation*}
\mathcal{A}(x,s)=\sum_{k=2}^{\infty }k^{s}A_{k}(x)=\sum_{n}x^{n}\sum 
_{\substack{ 1 \leq i \leq k  \\ \gcd (i,k)=1  \\ e(k,i)=n}}k^{s}
\end{equation*}%
and, as with $\mathcal{F(}s,x),$ one can show that $\frac{\partial }{%
\partial s}\mathcal{A}(x,s)|_{s=1}$ can be obtained by differentiating the
series term-by-term for sufficiently small $x$.

\section{Entropy of the uniform $(m,d)$-measures}

\subsection{The entropy generating function for the $d=2,m=3$ case}

\label{sec:2,1}

In this first subsection we consider the case when $d=2$ and $m=3$. This
case will illustrate the key combinatorial ideas without the complications
that arise in the general case, making precise the relationship between $%
F_{k}(x)$ (the generating function of the $(m,d)$-graph), $A_{k}(x)$ (the
generating function for the Euclidean tree), and $H_{k}(x)$ (the entropy
generating function).

Figure \ref{3-2-whole-graph} shows the first few levels of the $(3,2)$-graph
associated with the $(3,2)$-measure. Each node has three children: a middle
child, whose frequency is the same as its parent, and a left and a right
child. The left child of a node $X$ is the right child of $X$'s left
neighbour, and thus its frequency is the sum of those of its parents. The
analogous situation holds for the right child. A node of frequency $k$
induces a column of frequency-$k$ nodes below it.

\begin{figure}[tbp]
\includegraphics[width=\textwidth]{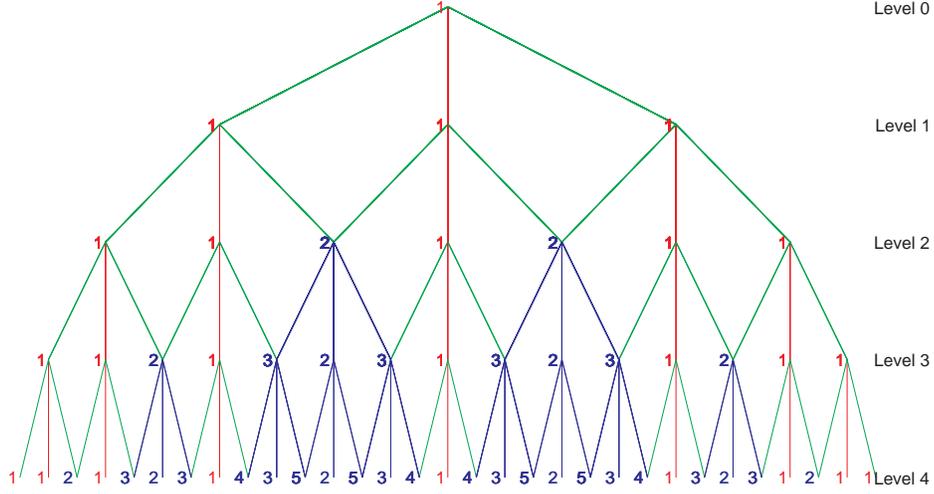}
\caption{The infinite self-similar graph associated with the $(3,2)$-measure}
\label{3-2-whole-graph}
\end{figure}

\subsubsection{Generating functions of subgraphs}

The first step is to partition the full $(3,2)$-graph into subgraphs we call
the $G$ and $P$-subgraphs. We will show that the $G$-subgraph is closely
related to the Euclidean tree and using this we will see how to compute its
generating functions. The $P$-subgraphs turn out to be very simple in this
particular case.

Note that the only frequency-one nodes in the graph are on the left and
right arcs descending from the top node and in the columns below the nodes
in those arcs. It is easy to see that these columns of ones partition the
graph into an infinite number of copies of the subgraph depicted in Figure %
\ref{3-2-subgraph}, with two copies starting at each level. We will call the
subgraphs between these column of ones the $\emph{G}$\emph{-subgraphs}.
Note, for example, that there are no nodes of the $G$-subgraph at relative
level 0, one node of weight 2 at level 1 and three nodes (two of weight 3
and one of weight 2) at level 2.

We will let $G_{k}(x)$ be the generating function for the number of nodes of
weight $k$ at level $n$ of the $G$-subgraph. For example $%
G_{5}(x)=2x^{3}+4x^{4}+\dots $.

The column of ones that divide various $G$-subgraphs will be called the 
\emph{P-subgraphs}. The generating functions for the $P$-subgraphs are
simply 
\begin{equation*}
P_{1}(x)=1+x+x^{2}+\cdots =\frac{1}{1-x}\text{ and }P_{k}(x)=0\text{ for }%
k\geq 2.
\end{equation*}

There is one $P$-subgraph starting at level 0 and two $P$-subgraphs starting
at each level $n\geq 1$. In addition, there are two $G$-subgraphs starting
at level $n$ for all $n\geq 1$. This gives the relationship 
\begin{equation}
F_{k}(x)=P_{k}(x)+\frac{2x}{1-x}P_{k}(x)+\frac{2x}{1-x}G_{k}(x).  \label{FPG}
\end{equation}

\begin{figure}[tbp]
\includegraphics[width=\textwidth]{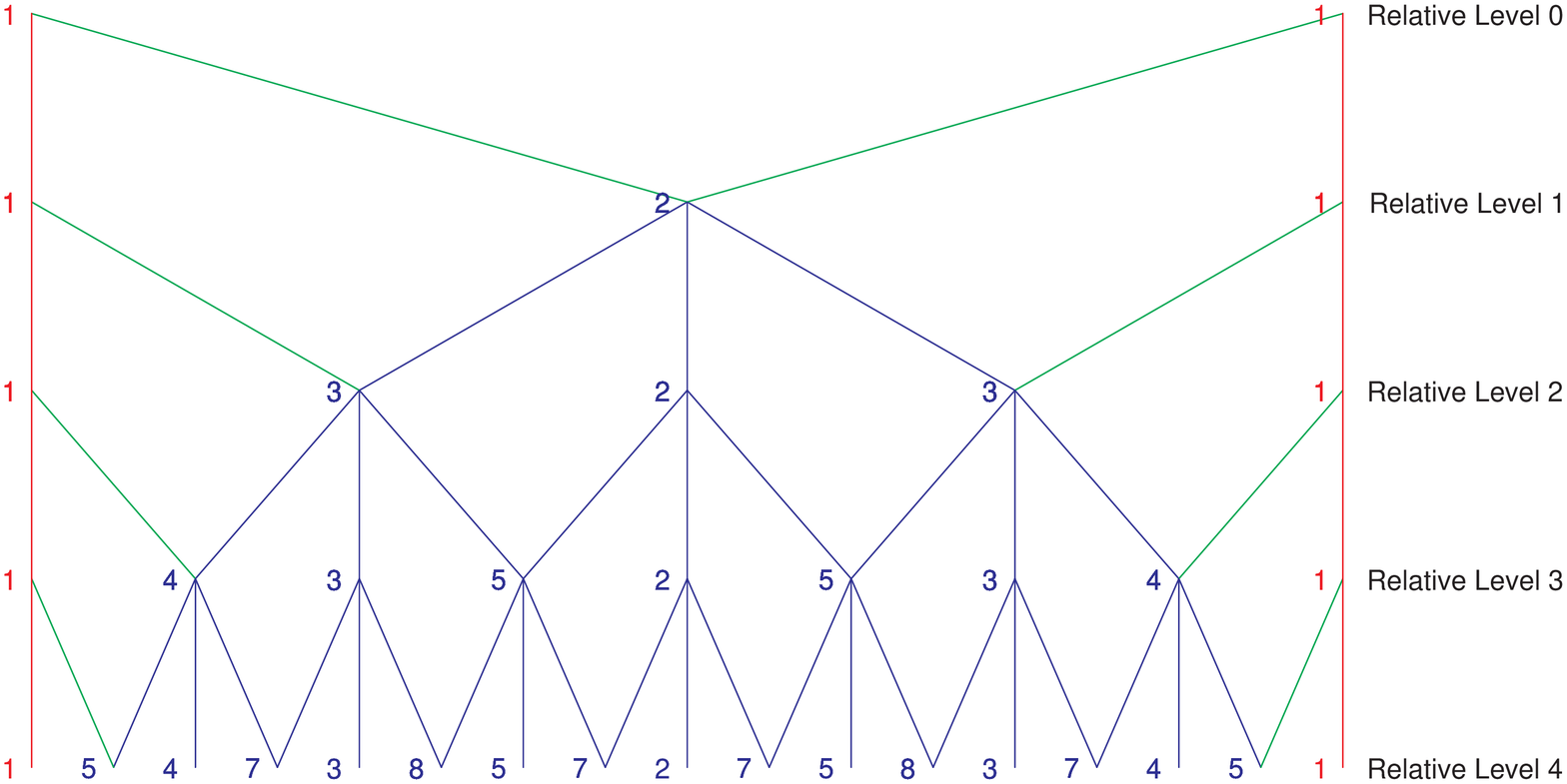}
\caption{The $G$-subgraph (blue) and $P$-subgraphs (red) bounding it}
\label{3-2-subgraph}
\end{figure}

The $G$-subgraph essentially consists of two copies of the dual graph of the
Euclidean tree, as we now explain. Note that the $G$-subgraph is symmetric
about the middle column of twos. If we take the dual graph of one of its
halves, and label each resulting node with the pair of nodes adjacent to it
in the $G$-subgraph (see Figure \ref{3-2-subgraph-with-euc-tree}), then we
get the Euclidean tree. Figure \ref{3-2-subgraph-with-euc-tree} gives the
first few levels of the Euclidean tree, as well as demonstrating its duality
with the $G$-subgraph.

\begin{figure}[tbp]
\begin{center}
\includegraphics[width=\textwidth]{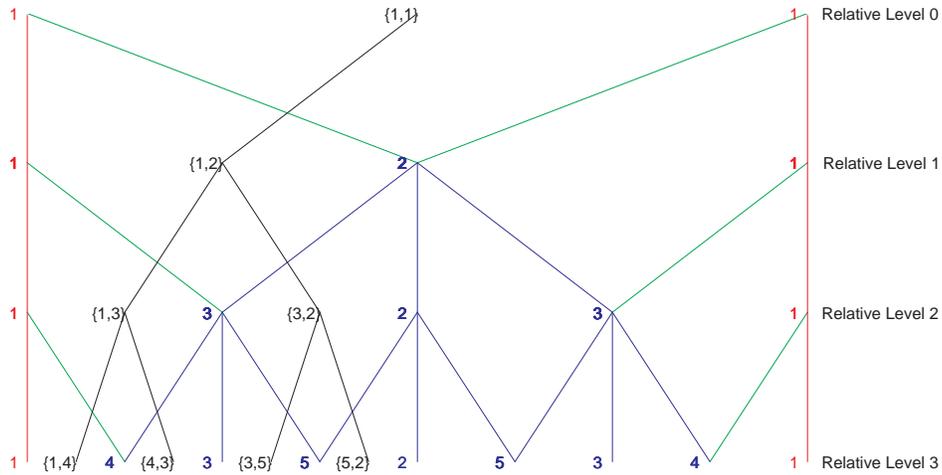}
\end{center}
\caption{The $G$-subgraph (blue) and $P$-subgraphs (red) bounding it and the
Euclidean tree dual (black)}
\label{3-2-subgraph-with-euc-tree}
\end{figure}

Since the generating functions for larger \textit{and} smaller labels in the
Euclidean tree is given by $\frac{1}{1-x}A_{k}(x)$ for $k\geq 2$, as
explained in the previous section, the duality relationship between the
Euclidean tree and the $G$-subgraph implies that the generating function for
the $G$-subgraph is 
\begin{equation*}
G_{1}(x)=0\text{, }G_{k}(x)=\frac{1}{1-x}A_{k}(x)\text{ for }k\geq 2\text{.}
\end{equation*}%
Combining this with equation (\ref{FPG}) shows that for $k\geq 2$, 
\begin{equation}
F_{k}(x)=P_{k}(x)+\frac{2x}{1-x}P_{k}(x)+\frac{2x}{(1-x)^{2}}A_{k}(x)=\frac{%
2x}{(1-x)^{2}}A_{k}(x).  \label{FPA}
\end{equation}

\subsubsection{The analytic extension of the entropy generating function}

\begin{theorem}
\label{thm:T} Let $H_{3,2}(x)=\sum_{n=1}^{\infty }h_{3,2}(n)x^{n}$ be the
generating function for the entropies of the levels of the $(3,2)$-graph.
There exists a function $T_{3,2}(x)$, analytic on a disk of radius $3$ about 
$0$, such that 
\begin{equation*}
H_{3,2}(x)=\frac{x}{(x-1)^{2}}T_{3,2}(x).
\end{equation*}
\end{theorem}

\begin{cor}
\label{cor:T} With $T_{3,2}(x)$ defined as above, we have 
\begin{equation*}
\mathfrak{H}_{3,2}=T_{3,2}(1).
\end{equation*}
\end{cor}

\begin{proof}
Let $U=\{x:|x|<3\}$. As $T(x)=T_{3,2}(x)$ is analytic on $U$, we see that $%
H(x)=H_{3,2}(x)$ is analytic on $U\setminus \{1\}$. As $H(x)=\frac{x}{%
(x-1)^{2}}T(x)$ and $T(x)$ is analytic on a disk of radius 2 around $1$,
there must exist coefficients $c(n)$ and $d(n)$ such that 
\begin{align*}
H(x)& =\sum_{n=0}^{\infty }h(n)x^{n}=x\sum_{n=-2}^{\infty }c(n)(x-1)^{n} \\
& =x\left( \frac{T(1)}{(x-1)^{2}}+\frac{T^{\prime }(1)}{x-1}%
+\sum_{n=0}^{\infty }c(n)(x-1)^{n}\right) \\
& =x\left( \frac{T(1)}{(x-1)^{2}}+\frac{T^{\prime }(1)}{x-1}%
+\sum_{n=0}^{\infty }d(n)x^{n}\right)
\end{align*}%
Since $H(x)$ is analytic on $U\setminus \{1\},$ we see that $%
\sum_{n=0}^{\infty }d(n)x^{n}$ is analytic on $U$ and hence $d(n)\rightarrow
0$ as $n\rightarrow \infty $. Further, $h(n)=T(1)n-T^{\prime }(1)+d(n-1)$,
whence 
\begin{equation*}
\mathfrak{H}_{3,2}=\lim_{n\rightarrow \infty }\frac{h(n)}{n\log _{2}2}=T(1).
\end{equation*}
\end{proof}

\begin{proof}[Proof of Theorem \protect\ref{thm:T}]
We remind the reader that $\mathcal{A}(x,s)=\sum_{k=2}^{\infty
}k^{s}A_{k}(x) $ and $\mathcal{F}(x,s)=\sum_{k=2}^{\infty }k^{s}F_{k}(x)$.
As these sums begin with $k=2$ and $P_{k}=0$ for $k\geq 2$, equation (\ref%
{FPA}) shows 
\begin{equation}
\mathcal{F}(x,s)=\frac{2x}{(1-x)^{2}}\mathcal{A}(x,s),  \label{DiffFA}
\end{equation}%
while differentiating the series $\sum k^{s}F_{k}(x)$ term-by-term with
respect to $s$ gives%
\begin{eqnarray*}
\frac{\partial }{\partial s}F(x,s) &=&\sum_{k}k^{s}\ln kF_{k}(x) \\
&=&\sum_{k}k^{s}\ln k\sum_{n}f(n,k)x^{n}.
\end{eqnarray*}%
From (\ref{hsimp}), we have 
\begin{equation*}
h(n)=n\log _{2}3-3^{-n}\sum_{k=1}^{\infty }f(n,k)k\log _{2}k,
\end{equation*}%
thus

\begin{align*}
H(x)& =\sum_{n=0}^{\infty }h(n)x^{n}=\sum_{n=0}^{\infty }\left( n\log
_{2}3-3^{-n}\sum_{k=1}^{\infty }f(n,k)k\log _{2}k\right) x^{n} \\
& =\sum_{n=0}^{\infty }nx^{n}\log _{2}3-\frac{1}{\ln 2}\left. \frac{\partial 
}{\partial s}\mathcal{F}(x/3,s)\right\vert _{s=1} \\
& =\frac{x}{(1-x)^{2}}\log _{2}3-\frac{2x}{3(1-x/3)^{2}\ln 2}\left. \frac{%
\partial }{\partial s}\mathcal{A}(x/3,s)\right\vert _{s=1}
\end{align*}%
where the final equality simply follows from (\ref{DiffFA}). Differentiating
the series $\sum_{k}k^{s}A_{k}(x)$, simplifying and using the definition of $%
A_{k}$ yields%
\begin{align*}
H(x)& =\frac{x}{(1-x)^{2}}\log _{2}3-\frac{2x}{3(1-x/3)^{2}}%
\sum_{k=2}^{\infty }k\log _{2}kA_{k}\left( \frac{x}{3}\right) \\
& =\frac{x}{(1-x)^{2}}\left( \log _{2}3-\frac{2(1-x)^{2}}{3(1-x/3)^{2}}%
\sum_{k=2}^{\infty }k\log _{2}kA_{k}\left( \frac{x}{3}\right) \right) \\
& =\frac{x}{(1-x)^{2}}\left( \log _{2}3-\frac{2(1-x)^{2}}{3(1-x/3)^{2}}%
\sum_{n=1}^{\infty }\left( \frac{x}{3}\right) ^{n}\sum_{\substack{ k>i,\gcd
(i,k)=1  \\ e(k,i)=n}}k\log _{2}k\right) .
\end{align*}

Finally, putting 
\begin{equation*}
L(x)=\sum_{n=1}^{\infty }\ell (n)x^{n}=(1-3x)^{2}\sum_{n=1}^{\infty
}x^{n}\sum_{\substack{ k>i,\gcd (i,k)=1  \\ e(k,i)=n}}k\log _{2}k,
\end{equation*}%
we conclude that 
\begin{equation*}
H(x)=\frac{x}{(1-x)^{2}}\left( \log _{2}3-\frac{2}{3(1-x/3)^{2}}%
L(x/3)\right) .
\end{equation*}

By \cite{gkt}, if $L(x)=\sum_{n=1}^{\infty }\ell (n)x^{n}$, then for $n\geq
3 $, $|\ell (n)|\leq \frac{2}{15\ln 2}$. This implies that $%
\sum_{n=1}^{\infty }\ell (n)x^{n}$ converges on the unit disk in the complex
plane, hence $H(x)$ has an analytic continuation to all $|x|<3$, as
required. Letting 
\begin{equation*}
T(x)=\log _{2}3-\frac{2}{3(1-x/3)^{2}}L(x/3)
\end{equation*}%
gives the desired result.
\end{proof}

\subsection{The entropy generating functions for the general uniform $(m,d)$%
-measure}

\label{sec:d,r}

\subsubsection{Generating functions of related subgraphs in the general case}

In the previous subsection, we used the fact that the $(3,2)$-graph can be
partitioned into a number of $G$-subgraphs and $P$-subgraphs and then showed
how the $G$-subgraph was related to the Euclidean tree. That allowed us to
find a generating function for the number of nodes of weight $k$ at level $n$
for the $(3,2)$-graph from which we developed the generating function for
the entropy.

In this subsection we will extend the notions of the $G$-subgraphs and $P$%
-subgraphs to the more general set up. Unfortunately, the $P$-subgraphs are
no longer simple as they will contain nodes with weights higher than 1. Both
graphs are still related to the Euclidean tree however, thus allowing us to
derive relations of $H(x)$ as before.

\begin{figure}[tbp]
\begin{center}
\includegraphics[width=\textwidth]{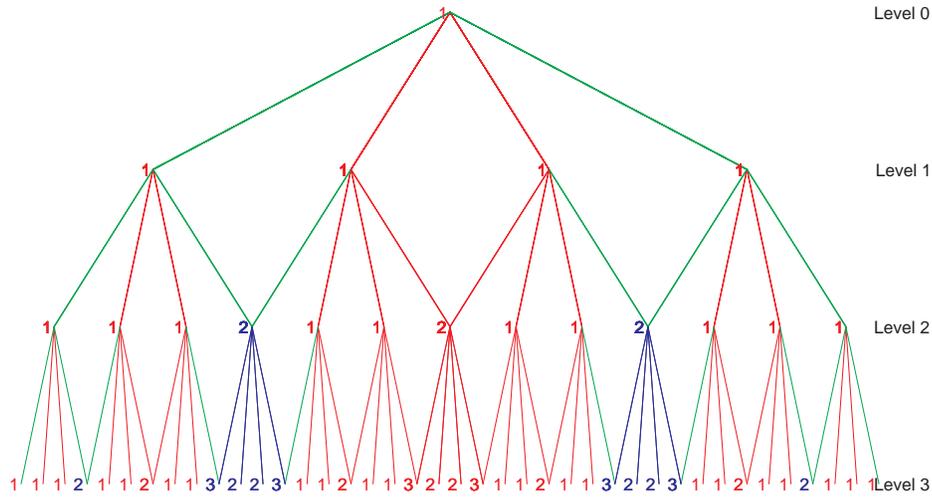}
\end{center}
\caption{The infinite self-similar graph associated with the $(4,3)$-measure}
\label{4-3-graph}
\end{figure}

\begin{figure}[tbp]
\begin{center}
\includegraphics[width=\textwidth]{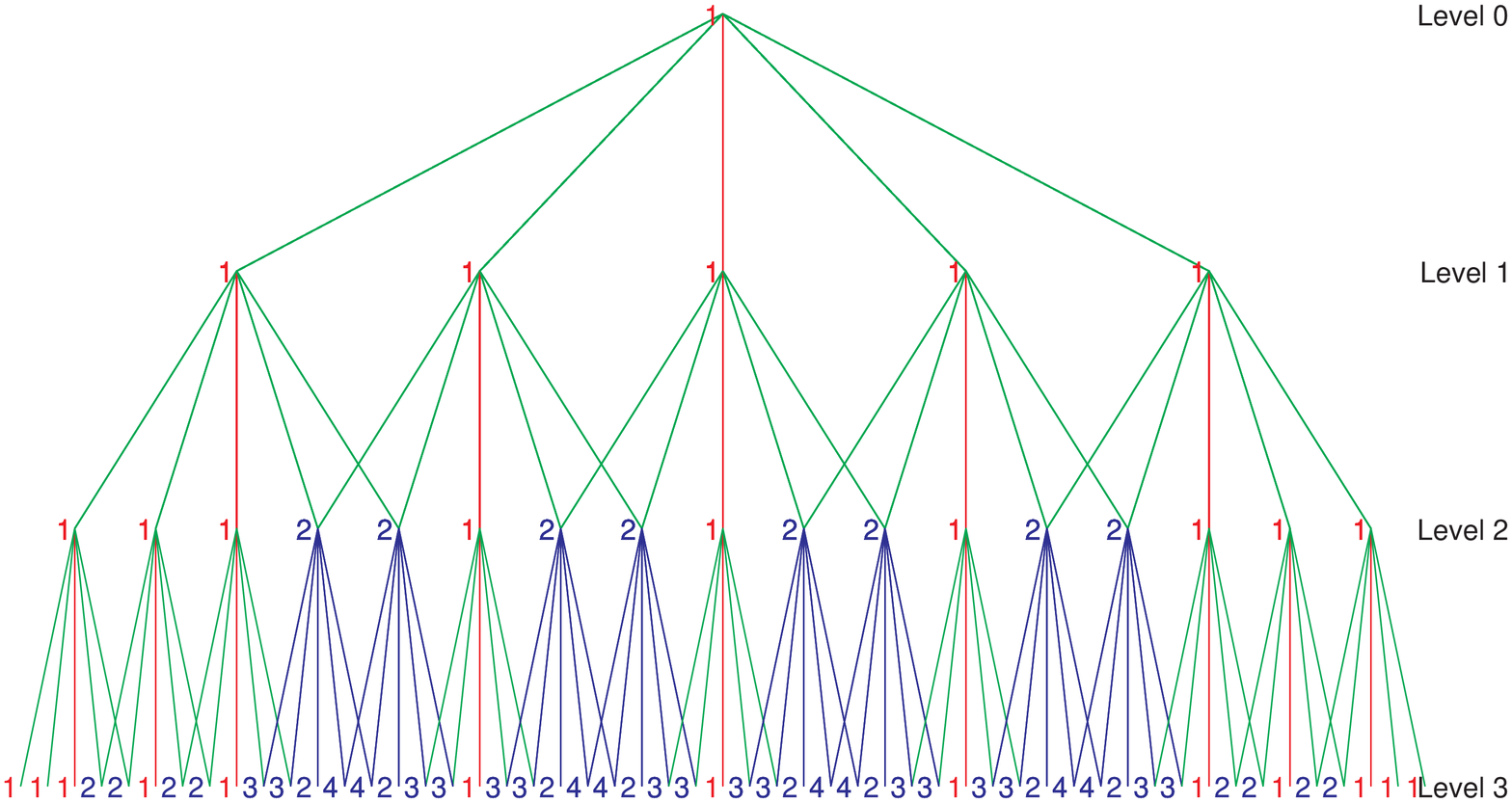}
\end{center}
\caption{The infinite self-similar graph associated with the $(5,3)$-measure}
\label{5-3-graph}
\end{figure}

\begin{figure}[tbp]
\begin{center}
\includegraphics[width=\textwidth]{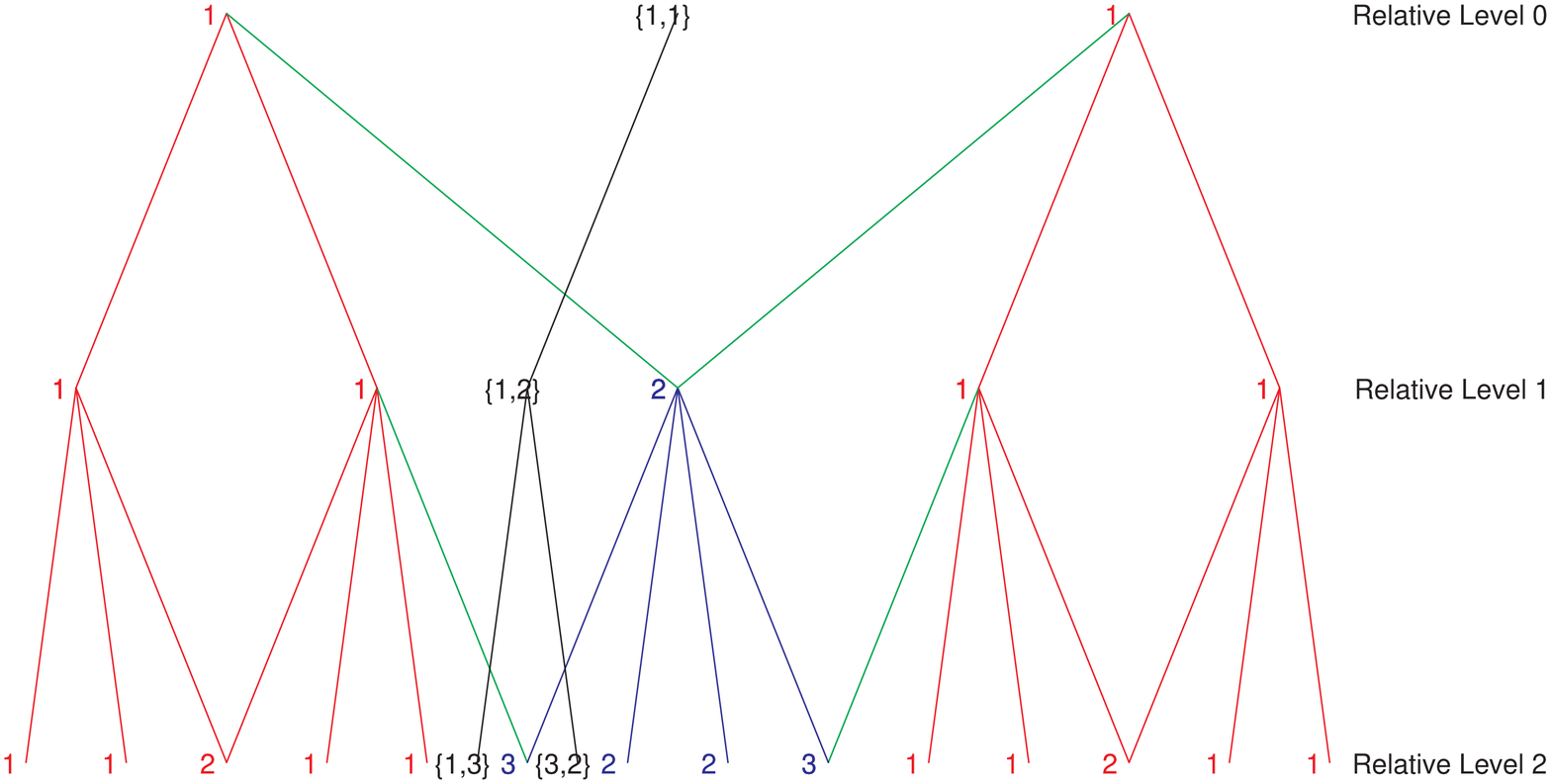}
\end{center}
\caption{The $G$-subgraph (blue) and $P$-subgraphs (red) bounding it and the
Euclidean graph dual (black)}
\label{4-3-subgraph-with-euc-tree}
\end{figure}

\begin{figure}[tbp]
\begin{center}
\includegraphics[width=\textwidth]{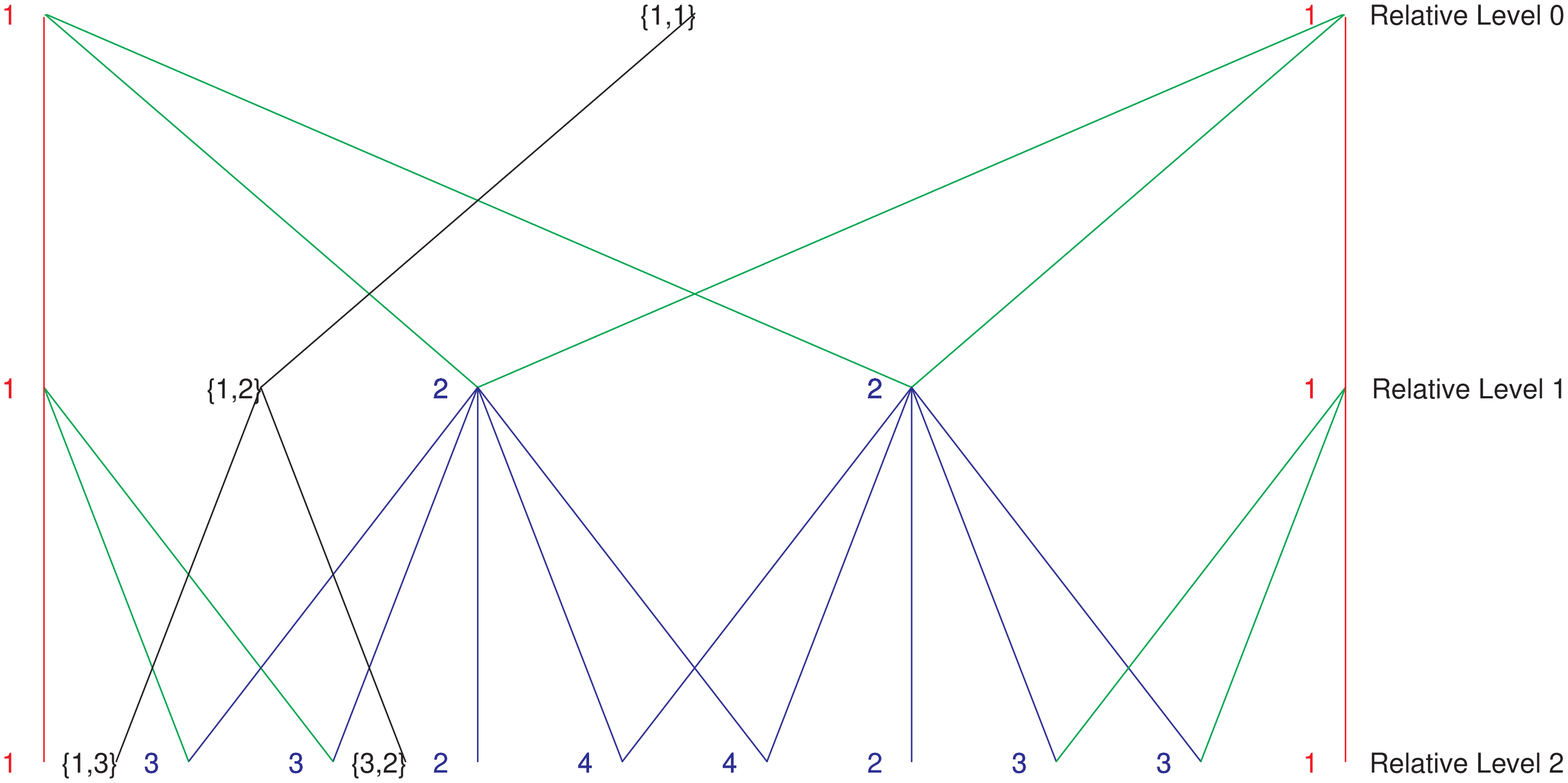}
\end{center}
\caption{The $G$-subgraph (blue) and $P$-subgraphs (red) bounding it and the
Euclidean graph dual (black)}
\label{5-3-subgraph-with-euc-tree}
\end{figure}

Figures \ref{4-3-graph} and \ref{5-3-graph} show the structure of the $(4,3)$
and $(5,3)$-graphs. For a general $(m,d)$-graph, each node has $m$ children,
and in general (assuming the parent has a neighbour on each side), the
leftmost and rightmost $r=m-d$ children will be \textquotedblleft
overlapping\textquotedblright\ and have a second parent, the first parent's
left or right neighbour. This is because $S_{i,j+d}=S_{i+1,j}$ for $%
j=0,\dots,r-1$ and otherwise, $S_{i,k}\neq S_{\ell ,j}$ since $m\leq 2d-1$.
This leaves $d-r$ non-overlapping children in the middle. The fact that two
nodes on the same level share a descendent if and only if they are adjacent
is the fundamental property that permits our analysis to work.

As before, we will partition the $(m,d)$-graph into $P$-subgraphs and $G$%
-subgraphs, but the definition of these two subgraphs will need to be
modified for this more general case. The $P$\emph{-subgraph} will begin with
a single node of weight $1$ at (relative) level 0. In the $(m,d)$-graph,
this node has $m=d+r$ children. We include in this $P$-subgraph all children
of this node, except the outer $r$ children on the left and on the right.
These inner, first level, children will always have weight $1$ (regardless
of the level of the original graph at which they begin). At the next level,
we consider again all children of these $d-r$ nodes, \textit{except} the
outer most $r$ right children of the right most node, and the outer most $r$
left children of the left most node. Note that some of these children will
have weight greater than $1$. We repeat this process for each lower level
with new $P$-subgraphs beginning on each level on these previously excluded
outer most nodes. We see that the outer most children of the $P$-subgraph
have weight $1$. Examples of $P$-subgraphs are given with red nodes in
Figures \ref{4-3-graph} and \ref{5-3-graph}. Notice that if $m=2d-1$ ($%
r=d-1) $, the $P$-subgraph is a single column of ones, as in the previous
section.

We define the $G$-subgraph to consist of the nodes between two adjacent $P$%
-subgraphs (not necessarily arising on the same level). As before, the $G$%
-subgraph has no nodes at relative (to the $P$-subgraphs) level 0. It will
have $r$ nodes at relative level $1$ of weight $2$. Examples of $G$%
-subgraphs are given with blue nodes in Figures \ref{4-3-graph} and \ref%
{5-3-graph}.

The generating function, $F_{k},$ of nodes of weight $k$ at level $n$ for
the $(m,d)$-graph, can again be written in terms of the generating functions
of the $G$-subgraphs and $P$-subgraphs. Indeed, we see that there is a
single $P$-subgraph starting at level $0$, $2r$ $P$-subgraphs starting at
level $1$ and, more generally, there are $2r$ $P$-subgraphs starting at
every level $n\geq 1$. Between each of pair of $P$-subgraphs there is a $G$%
-subgraph. Thus there are $2r$ $G$-subgraphs starting at every level $n\geq
1 $. This gives us the relationship 
\begin{equation}
F_{k}(x)=P_{k}(x)+\frac{2rx}{1-x}P_{k}(x)+\frac{2rx}{1-x}G_{k}(x)
\label{FPG general}
\end{equation}%
(which, of course, coincides with (\ref{FPG}) in the case $d=2$, $m=3$).

Having defined the $P$ and $G$-subgraphs, we now determine their generating
functions. First, consider a $P$-subgraph in the special case $r=d-1$. Then
the generating function is the same as before, namely $P_{1}(x)=\frac{1}{1-x}
$, $P_{k}(x)=0$ for $k\geq 2$.

If $r\neq d-1$, then the generating function is more complicated. The $P$%
-subgraph has a single node of weight $1$ at level $0$ and $d-r$ children at
level $1$ of weight $1$. These children can also be viewed as the starting
node of their own $P$-subgraph. Between each of these $d-r$ children there
is a $G$-subgraph. This gives us the relations 
\begin{align*}
P_{1}(x)& =1+(d-r)xP_{1}(x), \\
P_{k}(x)& =(d-r)xP_{k}(x)+(d-r-1)xG_{k}(x)\text{ for }k\geq 2.
\end{align*}%
Note these coincide with the equations given above in the special case $%
r=d-1 $ and simplify to 
\begin{equation}
P_{1}(x)=\frac{1}{1-(d-r)x},\text{ }P_{k}(x)=\frac{(d-r-1)x}{1-(d-r)x}%
G_{k}(x)\text{ for }k\geq 2.  \label{eq:Pk general 2}
\end{equation}

Now, consider the $G$-subgraph. As before, there is a relationship between
the $G$-subgraph and the Euclidean graph. To be more precise, there is a
relationship between the generating function for $G_{k}(x)$ and for $%
A_{k}(x) $. Consider a node in the Euclidean graph $\{a,b\}$ at level $n$
with children $\{a,a+b\}$ and $\{a+b,b\}$ at level $n+1$. Between these two
children there are $r$ nodes of weight $a+b$. Each of these nodes can be
thought of as the top node of an $(a+b)$ multiple of a $P$-subgraph. In
particular this means that the number of nodes of weight $k$ in one of these 
$(a+b)$ multiples of a $P$-subgraph is the same as the number of weight $%
k/(a+b)$ nodes in a $P$-subgraph. Between each of these $(a+b)$ multiples of 
$P$-subgraphs (of which there are $r$), there is a $(a+b)$ multiple of a $G$%
-subgraph, and there are $(r-1)$ such $G$-subgraphs. Similarly, the number
of nodes of weight $k$ in one of these $(a+b)$ multiples of a $G$-subgraph
is the same as the number of weight $k/(a+b)$ nodes in a $G$-subgraph. This
gives us the equations $G_{1}(x)=0$ and 
\begin{equation}
G_{k}(x)=\sum_{\ell \mid k,\text{ }\ell \neq 1}A_{\ell }(x)\left( rP_{k/\ell
}(x)+(r-1)G_{k/\ell }(x)\right) \text{ for }k\geq 2.  \label{eq:Gk general}
\end{equation}%
Observe that when $m=3,d=2$, this simplifies to $G_{k}(x)=\frac{1}{1-x}%
A_{k}(x)$, as we obtained before.

\subsubsection{The analytic extension of the entropy generating function for
the general case}

One of the main steps in proving Theorem \ref{thm:T} was to find a formula
for $F_{k}(x)$ in terms of only $A_{k}(x)$. Before doing this in the more
general case, we need to find an additional relationship.

\begin{lemma}
With the notation as above, we have 
\begin{equation*}
\left. \frac{\partial }{\partial s}\mathcal{F}(x,s)\right\vert _{s=1}=\frac{%
rx(m-1)(3x-1)^{2}}{(mx-1)^{2}(x-1)^{2}}\left. \frac{\partial }{\partial s}%
\mathcal{A}(x,s)\right\vert _{s=1}.
\end{equation*}
\end{lemma}

\begin{proof}
Let%
\begin{equation*}
\mathcal{G}(x,s)=\sum_{k=2}^{\infty }k^{s}G_{k}(x)\text{ and }\mathcal{P}%
(x,s)=\sum_{k=2}^{\infty }k^{s}P_{k}(x).
\end{equation*}

First, assume that $m=2d-1,$ $r=d-1$. The definitions of $\mathcal{G}$ and $%
\mathcal{A}$, equation (\ref{eq:Gk general}) and the fact that $G_{1}(x)=0$
imply that 
\begin{align*}
\mathcal{G}(x,s)\mathcal{A}(x,s)& =\sum_{k=2,\ell =2}(k\ell )^{s}A_{\ell
}(x)G_{k}(x) \\
& =\sum_{k=1,\ell =2}(k\ell )^{s}A_{\ell }(x)G_{k}(x) \\
& =\sum_{n=2}^{\infty }\sum_{\ell |n,\ell \neq 1}n^{s}A_{\ell }(x)G_{n/\ell
}(x) \\
& =\sum_{n=2}^{\infty }n^{s}\left( \sum_{\ell |n,\ell \neq 1}A_{\ell
}(x)G_{n/\ell }(x)+\frac{rA_{\ell }(x)P_{n/\ell }(x)}{r-1}\right)
-\sum_{n=2}^{\infty }n^{s}\sum_{\ell |n,\ell \neq 1}\frac{rA_{\ell
}P_{n/\ell }(x)}{r-1}
\end{align*}%
Since $P_{1}(x)=1/(1-x)$ and $P_{k}(x)=0$ for all $k\geq 2,$ this simplifies
to 
\begin{eqnarray*}
\mathcal{G}(x,s)\mathcal{A}(x,s) &=&\frac{1}{r-1}\sum_{n=2}^{\infty
}n^{s}G_{n}(x)-\frac{r}{(r-1)(1-x)}\sum_{n=2}^{\infty }n^{s}A_{n} \\
&=&\frac{1}{r-1}\mathcal{G}_{(}x,s)-\frac{r}{(r-1)(1-x)}\mathcal{A}(x,s).
\end{eqnarray*}%
Solving for $\mathcal{G}(x,s)$ gives 
\begin{equation*}
\mathcal{G}(x,s)=\frac{r\mathcal{A}(x,s)}{(1-x)(1-(r-1)\mathcal{A}(x,s))}
\end{equation*}%
and therefore from equation (\ref{FPG general}) we deduce that 
\begin{equation*}
\mathcal{F}(x,s)=\sum_{k=2}^{\infty }k^{s}G_{k}(x)\frac{2rx}{1-x}=\frac{%
2r^{2}x\mathcal{A}(x,s)}{(1-x)^{2}(1-(r-1)\mathcal{A}(x,s))}.
\end{equation*}

It follows from \cite{az} that $\mathcal{A}(x,1)=\frac{2x}{1-3x}$, hence a
straightforward calculation gives%
\begin{equation*}
\left. \frac{\partial }{\partial s}\mathcal{F}(x,s)\right\vert _{s=1}=\frac{%
2r^{2}x(3x-1)^{2}}{(1-x)^{2}(mx-1)^{2}}\left. \frac{\partial }{\partial s}%
\mathcal{A}(x,s)\right\vert _{s=1},
\end{equation*}%
which is the desired result in this special case.

In a similar fashion, one can verify that if $r\neq d-1$, then%
\begin{eqnarray*}
\mathcal{P}(x,s) &=&\mathcal{G}(x,s)\beta (x)\text{ and} \\
\mathcal{P}(x,s)\mathcal{A}(x,s) &=&\alpha (x)\mathcal{G}(x,s)-\alpha
(x)rP_{1}(x)\mathcal{A}(x,s)
\end{eqnarray*}%
where 
\begin{equation*}
\alpha (x)=\frac{(d-r-1)x}{(d-2r)x+r-1}\text{ and }\beta (x)=\frac{(d-r-1)x}{%
1-(d-r)x}.
\end{equation*}%
It follows from this that%
\begin{equation*}
\mathcal{F}(x,s)=\frac{\alpha (x)rP_{1}(x)\gamma (x)\mathcal{A}(x,s)}{\alpha
(x)/\beta (x)-\mathcal{A}(x,s)}
\end{equation*}%
where%
\begin{equation*}
\gamma (x)=1+\frac{2rx}{1-x}+\frac{2rx}{\beta (x)(1-x)}.
\end{equation*}%
Taking partial derivatives and evaluating at $s=1$ gives the claimed result.
\end{proof}

\begin{theorem}
\label{thm:T general} Let $m$ and $d$ be integers with $2\leq d<m\leq 2d-1$.
Let $H_{m,d}(x)=\sum_{n=1}^{\infty }h_{m,d}(n)x^{n}$ be the generating
function for the entropies of the levels of the $(m,d)$-graph. There exists
a function $T_{m,d}(x)$, analytic on a disk of radius $m$ about $0$, such
that 
\begin{equation*}
H_{m,d}(x)=\frac{x}{(x-1)^{2}}T_{m,d}(x).
\end{equation*}
\end{theorem}

We can generalize Corollary \ref{cor:T} in the obvious way to give

\begin{cor}
\label{cor:T general} With $T_{m,d}(x)$ defined as above, we have 
\begin{equation*}
\mathfrak{H}_{m,d}=\frac{T_{m,d}(1)}{\log _{2}d}.
\end{equation*}
\end{cor}

\begin{proof}[Proof of Theorem \protect\ref{thm:T general}]
Again, we begin by recalling that 
\begin{align*}
H(x) =H_{m,d}(x)& =\sum_{n=0}^{\infty }\left( n\log
_{2}m-m^{-n}\sum_{k=1}^{\infty }f(n,k)k\log _{2}k\right) x^{n} \\
&=\frac{x}{(1-x)^{2}}\log _{2}m-\frac{1}{\ln 2}\left. \frac{\partial }{%
\partial s}\mathcal{F}(x/m,s)\right\vert _{s=1}.
\end{align*}

Put 
\begin{equation*}
R(x)=\frac{r(m-1)m}{(m-x)^{2}}.
\end{equation*}%
Using the formula we obtained for $\frac{\partial }{\partial s}\mathcal{F}%
(x,s)$ in the previous lemma and differentiating $\mathcal{A}(x/m,s)$
term-by-term, gives%
\begin{eqnarray*}
H(x) &=&\frac{x}{(1-x)^{2}}\left( \log _{2}m-\frac{1}{\ln 2}R(x)\frac{%
(m-3x)^{2}}{m^{2}}\left. \frac{\partial }{\partial s}\mathcal{A}%
(x/m,s)\right\vert _{s=1}\right) \\
&=&\frac{x}{(1-x)^{2}}\left( \log _{2}m-R(x)\frac{(m-3x)^{2}}{m^{2}}%
\sum_{k=2}^{\infty }k\log _{2}kA_{k}(x/m)\right) .
\end{eqnarray*}%
As in the previous theorem, set%
\begin{equation*}
L(x)=\sum_{n=1}^{\infty }\ell (n)x^{n}=(1-3x)^{2}\sum_{n=1}^{\infty
}x^{n}\sum_{\substack{ k>i,\gcd (i,k)=1  \\ e(k,i)=n}}\log _{2}k,
\end{equation*}%
whence 
\begin{equation*}
H(x)=\frac{x}{(1-x)^{2}}\left( \log _{2}m-R(x)L(\frac{x}{m})\right) .
\end{equation*}

As $L(x)$ is analytic on the unit disk, $H(x)$ has an analytic continuation
to $|x|<m$, as claimed. Letting $T(x)=\log _{2}m-R(x)L(x/m)$ completes the
proof.
\end{proof}

In Section \ref{sec:comp}, we will use this generating function to extract
an entropy estimate and an error bound for the uniform $(m,d)$-measures and
give explicit numerical results for the case $2\leq d<m\leq 10$.

\section{Bounds for the entropy for the non-uniform $(m,d)$-measures}

In this section we consider the non-uniform $(m,d)$-measures. Recall that
the Garsia entropy is given by (see (\ref{entropymu}))%
\begin{equation*}
\mathfrak{H}_{\mu }=\lim_{n\rightarrow \infty }\frac{h_{\mu }(n)}{n\log _{2}d%
}=\lim_{n\rightarrow \infty }\frac{-\sum_{p\in W_{n}}p\log _{2}p}{n\log _{2}d%
}.
\end{equation*}

The goal of this section is to prove

\begin{proposition}
\label{biased}If $\mu $ is the $(m,d)$-measure associated with probabilities 
$\{p_{i}\}_{i=0}^{m-1}$, then%
\begin{equation*}
\left( \log _{2}d\right) \mathfrak{H}_{\mu }\in \left[ -%
\sum_{i=0}^{m-1}p_{i}\log
_{2}p_{i}-\sum_{i=0}^{m-d-1}(p_{i}+p_{d+i}),-\sum_{i=0}^{m-1}p_{i}\log
_{2}p_{i}\right] .
\end{equation*}
\end{proposition}

\begin{proof}
Set $\triangle h_{n}:=h_{\mu }(n)-h_{\mu }(n-1)$ for $n\in \mathbb{N}$
(putting $h(0)=0$), so that%
\begin{equation*}
\mathfrak{H}_{\mu }=\lim_{n\rightarrow \infty }\frac{\triangle
h_{1}+\triangle h_{2}+\dots+\triangle h_{n}}{n\log _{2}d}.
\end{equation*}%
Bounds for $\triangle h_{n}$ will then give bounds for the entropy. By
definition, $\triangle h_{n}=\sum_{p\in W_{n}}(-p\log_{2}p)-\sum_{p\in
W_{n-1}}(-p\log _{2}p)$. Note that each node from level $n$ comes from one
or two nodes from level $n-1$. We partition $W_{n}$ accordingly into $I_{n}$
for those nodes coming from one node at level $n-1$, and $J_{n}$ for those
nodes coming from two. It is worth noting that the left most $m-d$ and right
most $m-d$ nodes at level  $n$ are in $I_n$ and not $J_n$. 
With this notation, 
\begin{equation*}
\triangle h_{n}=\sum_{p\in W_{n-1}}p\log _{2}p-\sum_{p\in I_{n}}p\log
_{2}p-\sum_{p\in J_{n}}p\log _{2}p.
\end{equation*}

Using the fact that $\sum_{i=0}^{m-1}p_{i}=1$, we can partition the first
term to pair with the last two to give: 
\begin{eqnarray}
\triangle h(n) &=& \sum_{p\in W_{n-1}} \sum_{i=0}^{m-1} p_i p\log
_{2}p-\sum_{p\in I_{n}}p\log_{2}p-\sum_{p\in J_{n}}p\log _{2}p  \notag \\
&=& \left( \sum_{p\in W_{n-1}}\sum_{i=m-d}^{d-1} p_i p\log_{2}p -\sum_{p\in
I_{n}}p\log _{2}p\right)  \label{eq:star} \\
&& + \left( \sum_{p\in W_{n-1}} \sum_{i=0}^{m-d-1} (p_i + p_{i+d})
p\log_{2}p-\sum_{p\in J_{n}}p\log _{2}p\right) .  \notag
\end{eqnarray}

With the exception of the right and left most $m-d$ nodes, each node in $%
I_{n}$ is obtained by multiplying a unique node in $W_{n-1}$ by some $p_{i}$
($m-d\leq i\leq d-1$). The left and right most $m-d$ nodes result from
multiplying $p_{0}^{n-1}$ by some $p_{i}$ for $0\leq i\leq m-d-1$, (for the
left most) and multiplying $p_{m-1}^{n-1}$ by $p_{i}$ for some $p_{i}$ for $%
d \leq i\leq m-1$, (for the right most). Thus using the fact that $%
\sum_{p\in W_{n-1}}p=1,$ the first term of equation \eqref{eq:star}
simplifies to 
\begin{eqnarray*}
\text{first term} &=&\sum_{p\in W_{n-1}}\sum_{i=m-d}^{d-1}p_{i}p\log
_{2}p-\sum_{p\in W_{n-1}}\sum_{i=m-d}^{d-1}p_{i}p\log _{2}p_{i}p \\
&&-\sum_{i=0}^{m-d-1}p_{i}p_{0}^{n-1}\log _{2}p_{i}p_{0}^{n-1}
  -\sum_{i=d}^{m-1}p_{i}p_{m-1}^{n-1}\log_{2}p_{i}p_{m-1}^{n-1} \\
&=&-\sum_{p\in W_{n-1}}\sum_{i=m-d}^{d-1}p_{i}p\log _{2}p_{i} \\
&&-\sum_{i=0}^{m-d-1}p_{i}p_{0}^{n-1}\log _{2}p_{i}p_{0}^{n-1}
  -\sum_{i=d}^{m-1}p_{i}p_{m-1}^{n-1}\log_{2}p_{i}p_{m-1}^{n-1} \\
&= & -\sum_{i=m-d}^{d-1}p_{i}\log _{2}p_{i}
     -\sum_{i=0}^{m-d-1}p_{i}p_{0}^{n-1}\log _{2}p_{i}p_{0}^{n-1}
     -\sum_{i=d}^{m-1}p_{i}p_{m-1}^{n-1}\log_{2}p_{i}p_{m-1}^{n-1} \\
\end{eqnarray*}

To deal with the second term, observe that each node in $J_{n}$ comes from
two adjacent nodes from level $n-1$, thus we can rewrite the second term of
equation \eqref{eq:star} as 
\begin{eqnarray}
\text{second term} 
& =&
\sum_{i=0}^{m-d-1}p_{i}p_{0}^{n-1}\log_{2}p_{0}^{n-1} + \sum_{i=d}^{m-1}
p_{i}p_{m-1}^{n-1}\log_{2}p_{m-1}^{n-1}  \notag \\
&& +\sum_{\substack{ (p,q)\text{ adjacent}  \\ \text{in level }n-1}}%
\sum_{i=0}^{m-d-1}p_{i}p\log _{2}p+p_{d+i}q\log _{2}q  \label{eq:star2} \\
&& -\sum_{\substack{ (p,q)\text{ adjacent}  \\ \text{in level }n-1}}%
\sum_{i=0}^{m-d-1}(pp_{i}+qp_{d+i})\log _{2}(pp_{i}+qp_{d+i})  \notag
\end{eqnarray}

We concentrate on the last two terms of equation \eqref{eq:star2}. First,
write that sum as 
\begin{eqnarray*}
\text{last two terms} &=&\sum_{\substack{ (p,q)\text{ adjacent}  \\ \text{in
level }n-1}}\sum_{i=0}^{m-d-1}\left( p_{i}p\log _{2}p+p_{d+i}q\log
_{2}q\right) \\
&&-\sum_{\substack{ (p,q)\text{ adjacent}  \\ \text{in level }n-1}}%
\sum_{i=0}^{m-d-1}(pp_{i}+qp_{d+i})\log _{2}(pp_{i}+qp_{d+i}) \\
&=&\sum_{\substack{ (p,q)\text{ adjacent}  \\ \text{in level }n-1}}%
\sum_{i=0}^{m-d-1} p_{i}p\log _{2}p_{i}p+p_{d+i}q\log _{2}p_{d+i}q \\
&&-\sum_{\substack{ (p,q)\text{ adjacent}  \\ \text{in level }n-1}}%
\sum_{i=0}^{m-d-1}(pp_{i}+qp_{d+i})\log _{2}(pp_{i}+qp_{d+i}) \\
&&-\sum_{\substack{ (p,q)\text{ adjacent}  \\ \text{in level }n-1}}%
\sum_{i=0}^{m-d-1}\left( p_{i}p\log _{2}p_{i}+p_{d+i}q\log _{2}p_{d+i}\right)
\\
&=&\sum_{\substack{ (p,q)\text{ adjacent}  \\ \text{in level }n-1}}%
\sum_{i=0}^{m-d-1}\left( p_{i}p\log _{2}\frac{p_{i}p}{p_{i}p+p_{d+i}q}%
+p_{d+i}q\log _{2}\frac{p_{d+i}q}{p_{i}p+p_{d+i}q}\right) \\
&&-\sum_{\substack{ (p,q)\text{ adjacent}  \\ \text{in level }n-1}}%
\sum_{i=0}^{m-d-1}(p_{i}p\log _{2}p_{i}+p_{d+i}q\log _{2}p_{d+i}).
\end{eqnarray*}%
If we let $D(x)=x\log _{2}x+(1-x)\log _{2}(1-x)$ for $x\in (0,1)$, and the
fact that every node appears twice in the sum over $(p,q)$  adjacent nodes
    at level $n-1$, except the first and last, then it is straightforward to
    check this simplifies to 
\begin{eqnarray*}
\text{last two terms} &=&\sum_{\substack{ (p,q)\text{ adjacent}  \\ \text{in
level }n-1}}\sum_{i=0}^{m-d-1}(p_{i}p+p_{d+i}q)D\left( \frac{p_{i}p}{%
p_{i}p+p_{d+i}q}\right) \\
&&-\sum_{\substack{ (p,q)\text{ adjacent}  \\ \text{in level }n-1}}%
\sum_{i=0}^{m-d-1}(p_{i}p\log _{2}p_{i}+p_{d+i}q\log _{2}p_{d+i}) \\
&=&\sum_{\substack{ (p,q)\text{ adjacent}  \\ \text{in level }n-1}}%
\sum_{i=0}^{m-d-1}(p_{i}p+p_{d+i}q)D\left( \frac{p_{i}p}{p_{i}p+p_{d+i}q}%
\right) \\
&&-\sum_{i=0}^{m-d-1}p_{i}\log _{2}p_{i}-\sum_{i=d}^{m-1} p_{i}\log _{2}p_{i}
\\
&&+\sum_{i=0}^{m-d-1}p_{0}^{n-1}p_{i}\log_{2}p_{i}+\sum_{i=d}^{m-1}
p_{m-1}^{n-1}p_{i}\log_{2}p_{i}.
\end{eqnarray*}

Putting the above together, we see that 
\begin{eqnarray*}
\triangle h_{n} &= & -\sum_{i=m-d}^{d-1}p_{i}\log _{2}p_{i}
-\sum_{i=0}^{m-d-1}p_{i}p_{0}^{n-1}\log _{2}p_{i}p_{0}^{n-1}
-\sum_{i=d}^{m-1}p_{i}p_{m-1}^{n-1}\log_{2}p_{i}p_{m-1}^{n-1} \\
&& + \sum_{i=0}^{m-d-1}p_{i}p_{0}^{n-1}\log_{2}p_{0}^{n-1} +
\sum_{i=d}^{m-1} p_{i}p_{m-1}^{n-1}\log_{2}p_{m-1}^{n-1}  \notag \\
&&+ \sum_{\substack{ (p,q)\text{ adjacent}  \\ \text{in level }n-1}}%
\sum_{i=0}^{m-d-1}(p_{i}p+p_{d+i}q)D\left( \frac{p_{i}p}{p_{i}p+p_{d+i}q}%
\right) \\
&&-\sum_{i=0}^{m-d-1}p_{i}\log _{2}p_{i}-\sum_{i=d}^{m-1} p_{i}\log
_{2}p_{i} +\sum_{i=0}^{m-d-1}p_{0}^{n-1}p_{i}\log_{2}p_{i}+\sum_{i=d}^{m-1}
p_{m-1}^{n-1}p_{i}\log_{2}p_{i}. \\
&= & -\sum_{i=m-d}^{d-1}p_{i}\log _{2}p_{i} -\sum_{i=0}^{m-d-1}p_{i}\log
_{2}p_{i}-\sum_{i=d}^{m-1} p_{i}\log _{2}p_{i} \\
&& -\sum_{i=0}^{m-d-1}p_{i}p_{0}^{n-1} \left ( \log _{2}p_{i}p_{0}^{n-1} -
\log_2 p_0^{n-1} - \log_2 p_i \right) \\
&& -\sum_{i=d}^{m-1}p_{i}p_{m-1}^{n-1} \left ( \log _{2}p_{i}p_{m-1}^{n-1} -
\log_2 p_m-1^{n-1} - \log_2 p_i \right) \\
&&+ \sum_{\substack{ (p,q)\text{ adjacent}  \\ \text{in level }n-1}}%
\sum_{i=0}^{m-d-1}(p_{i}p+p_{d+i}q)D\left( \frac{p_{i}p}{p_{i}p+p_{d+i}q}%
\right) \\
&= & -\sum_{i=0}^{m-1}p_{i}\log _{2}p_{i} + \sum_{\substack{ (p,q)\text{
adjacent}  \\ \text{in level }n-1}}\sum_{i=0}^{m-d-1}(p_{i}p+p_{d+i}q)D%
\left( \frac{p_{i}p}{p_{i}p+p_{d+i}q}\right) 
\end{eqnarray*}%
Let 
\begin{equation*}
a :=-\sum_{i=0}^{m-1}p_{i}\log _{2}p_{i}
\end{equation*}%
\newline
As before, every node appears twice in the sum over $(p,q)$ adjacent nodes
at level $n-1$, except the first and last, hence
\begin{eqnarray*}
\sum_{\substack{ (p,q)\text{ adjacent}  \\ \text{in level }n-1}}%
\sum_{i=0}^{m-d-1}(p_{i}p+p_{d+i}q) &=&\sum_{p\in
W_{n-1}}\sum_{i=0}^{m-d-1}(p_{i}+p_{d+i})p-%
\sum_{i=0}^{m-d-1}p_{i}p_{0}^{n-1}-\sum_{i=d}^{m-1}p_{i}p_{m-1}^{n-1} \\
&=&\sum_{i=0}^{m-d-1}(p_{i}+p_{d+i})-\sum_{i=0}^{m-d-1}p_{i}p_{0}^{n-1}-%
\sum_{i=d}^{m-1}p_{i}p_{m-1}^{n-1} \\
& := &b_{n}
\end{eqnarray*}

Since the range of the function $D$ is the interval $[-1,0]$, it follows that 
\begin{equation*}
\triangle h_{n}\in \lbrack a-b_{n},a]
\end{equation*}%
Of course, $a = -\sum_{i=0}^{m-1}p_{i}\log _{2}p_{i}$ and $b_{n}\rightarrow
\sum_{i=0}^{m-d-1}(p_{i}+p_{d+i})$ as $n\rightarrow \infty $, thus 
\begin{eqnarray*}
\mathfrak{H}_{\mu }\log _{2}d &=&\lim_{n\rightarrow \infty }\frac{h_{\mu }(n)%
}{n}=\lim_{n}\triangle h_{n} \\
&\in &\left[ -\sum_{i=0}^{m-1}p_{i}\log
_{2}p_{i}-\sum_{i=0}^{m-d-1}(p_{i}+p_{d+i}),-\sum_{i=0}^{m-1}p_{i}\log
_{2}p_{i}\right] .
\end{eqnarray*}
\end{proof}

\begin{remark}
We remark that the quantity $-\sum_{i=0}^{m-1}p_{i}\log _{2}p_{i}/\log d$ is
known as the similarity dimension of this measure and the similarity
dimension of a self-similar measure is always an upper bound for its
Hausdorff dimension. We also note that if the $p_{i}$ are suitably biased,
then $\left\vert \sum_{i=0}^{m-1}p_{i}\log _{2}p_{i}\right\vert $ is very
small, so the entropy less than $1$ and hence the measure is singular.
\end{remark}

In the next section we will illustrate this bound in some concrete examples.

\section{Entropy estimates and bounds}

\label{sec:comp}

\subsection{Uniform case}

For the uniform $(m,d)$-measure $\mu $, Corollary \ref{cor:T general} gives 
\begin{equation*}
\mathfrak{H}_{\mu }=\frac{T_{m,d}(1)}{\log _{2}d}\text{ where }%
T_{m,d}(x)=\log _{2}m-R(x)L(x/m),
\end{equation*}%
and the functions $R(x)$ and $L(x)=\sum_{n=1}^{\infty }\ell (n)x^{n}$ are as
given in the proof of Theorem \ref{thm:T general}. Since \cite{gkt} gives $%
\left\vert \ell (n)\right\vert \leq 2/(15\ln 2)$ for $n\geq 3$, we see that 
\begin{equation*}
\left\vert \sum_{n=N+1}^{\infty }\ell (n)\frac{1}{m^{n}}\right\vert \leq 
\frac{2m^{-N}}{15(m-1)\ln 2}.
\end{equation*}%
This allows us to determine, for each $\epsilon >0,$ an integer $N$ such
that $\mathfrak{H}_{\mu }$ differs from the sum of the first $N$ terms by at
most $\epsilon $. In Table \ref{tab:entropy} we have used this to compute
the entropy (equivalently, the Hausdorff dimension) of $\mu $ for $2\leq
d\leq 10$ and $d<m\leq 2d-1$ to 10 decimal points. We have also indicated
the integer $N$ that was necessary to perform this calculation. All these
measures are singular as their entropy is strictly less than one.

\begin{table}[tbp]
\begin{center}
\begin{tabular}{l|l|l|lll|l|l|l}
$d$ & $r$ & Entropy & $N$ &  & $d$ & $r$ & Entropy & $N$ \\ 
\cline{1-4}\cline{6-9}
2 & 1 & .9887658714 & 20 &  & 8 & 1 & .9847485173 & 9 \\ 
&  &  &  &  & 8 & 2 & .9774806174 & 9 \\ 
3 & 1 & .9696751053 & 15 &  & 8 & 3 & .9756417435 & 9 \\ 
3 & 2 & .9888495673 & 13 &  & 8 & 4 & .9775746034 & 8 \\ 
&  &  &  &  & 8 & 5 & .9821685970 & 8 \\ 
4 & 1 & .9723043945 & 13 &  & 8 & 6 & .9886592929 & 8 \\ 
4 & 2 & .9744950829 & 12 &  & 8 & 7 & .9965086797 & 8 \\ 
4 & 3 & .9917161717 & 11 &  &  &  &  &  \\ 
&  &  &  &  & 9 & 1 & .9865170224 & 8 \\ 
5 & 1 & .9763335645 & 11 &  & 9 & 2 & .9793377946 & 8 \\ 
5 & 2 & .9724991949 & 11 &  & 9 & 3 & .9766109550 & 8 \\ 
5 & 3 & .9798311869 & 10 &  & 9 & 4 & .9770870210 & 8 \\ 
5 & 4 & .9936600571 & 10 &  & 9 & 5 & .9798993303 & 8 \\ 
&  &  &  &  & 9 & 6 & .9844327917 & 8 \\ 
6 & 1 & .9797875450 & 10 &  & 9 & 7 & .9902423029 & 8 \\ 
6 & 2 & .9736047261 & 10 &  & 9 & 8 & .9970004104 & 8 \\ 
6 & 3 & .9759857840 & 9 &  &  &  &  &  \\ 
6 & 4 & .9837495163 & 9 &  & 10 & 1 & .9879592199 & 8 \\ 
6 & 5 & .9949548480 & 9 &  & 10 & 2 & .9810095410 & 8 \\ 
&  &  &  &  & 10 & 3 & .9777693162 & 8 \\ 
7 & 1 & .9825497418 & 9 &  & 10 & 4 & .9772748839 & 8 \\ 
7 & 2 & .9754969280 & 9 &  & 10 & 5 & .9788382244 & 8 \\ 
7 & 3 & .9751879641 & 9 &  & 10 & 6 & .9819582547 & 8 \\ 
7 & 4 & .9793642691 & 9 &  & 10 & 7 & .9862637671 & 7 \\ 
7 & 5 & .9865742717 & 9 &  & 10 & 8 & .9914757004 & 7 \\ 
7 & 6 & .9958552030 & 8 &  & 10 & 9 & .9973815856 & 7%
\end{tabular}%
\end{center}
\caption{Entropy of $(m,d)$-measures, to 10 decimal places}
\label{tab:entropy}
\end{table}

\subsection{Non-uniform case}

\begin{example}
Take $(m,d)=(3,2)$, $p_{0}=\frac{1}{t+2}$, $p_{1}=\frac{t}{t+2}$, $p_{2}=%
\frac{1}{t+2}$ and let $\mu $ be the corresponding self-similar measure. By
Prop. \ref{biased}, $\mathfrak{H}_{\mu }$ lies in the interval 
\begin{equation*}
\left[ \frac{-t\ln \frac{t}{t+2}-2\ln \frac{1}{t+2}}{\ln 2}-\frac{2}{t+2},%
\frac{-t\ln \frac{t}{t+2}-2\ln \frac{1}{t+2}}{\ln 2}\right] .
\end{equation*}%
This can be improved. Indeed, one can use an induction argument in this case
to show that if $p,q$ are adjacent nodes, then $1/(t+1)\leq p/(p+q)\leq
t/(t+1)$. Consequently, we may restrict the range of $D$ to $[D(\frac{1}{2}%
),D(\frac{1}{t+1})]$. With this improvement, for the values of $t=1,\dots
,10 $, we deduce that the entropy $\mathfrak{H}_{\mu }$ belongs to the
intervals given in Table \ref{tab:non-uniform}. We note that the first entry 
$1/3,1/3,1/3$ corresponds to the $d=2$, $r=1$ case of Table \ref{tab:entropy}%
.

Of course, the dimension of any measure $\mu $ on $\mathbb{R}$ is bounded
above by one, hence the upper bounds from this technique only give
meaningful bounds on the dimension after the fifth entry when they establish
that these measures are singular.

As the lower bound for the entropy for the $\frac{1}{4},\frac{1}{2},\frac{1}{%
4}$ measure is one, this measure has dimension one. In fact, this measure is
the convolution $m\ast m$ where $m$ is Lebesgue measure restricted to $[0,1]$%
.

\begin{table}[tbp]
\begin{center}
\begin{tabular}{l|l|l}
$p_0, p_1, p_2$ & Lower bound & Upper bound \\ \hline
$\frac{1}{3}, \frac{1}{3}, \frac{1}{3} $ & .9182958344 & 1.584962501 \\ 
$\frac{1}{4}, \frac{2}{4}, \frac{1}{4} $ & 1. & 1.040852083 \\ 
$\frac{1}{5}, \frac{3}{5}, \frac{1}{5} $ & .9709505935 & 1.046439344 \\ 
$\frac{1}{6}, \frac{4}{6}, \frac{1}{6} $ & .9182958336 & 1.010986469 \\ 
$\frac{1}{7}, \frac{5}{7}, \frac{1}{7} $ & .8631205682 & .9631141620 \\ 
$\frac{1}{8}, \frac{6}{8}, \frac{1}{8} $ & .8112781250 & .9133599301 \\ 
$\frac{1}{9}, \frac{7}{9}, \frac{1}{9} $ & .7642045081 & .8656346320 \\ 
$\frac{1}{10}, \frac{8}{10}, \frac{1}{10} $ & .7219280941 & .8212764285 \\ 
$\frac{1}{11}, \frac{9}{11}, \frac{1}{11} $ & .6840384354 & .7805846910 \\ 
$\frac{1}{12}, \frac{10}{12}, \frac{1}{12}$ & .6500224217 & .7434395905%
\end{tabular}%
\end{center}
\caption{Entropy of non-uniform Cantor-like measures}
\label{tab:non-uniform}
\end{table}
\end{example}

\providecommand{\bysame}{\leavevmode\hbox to3em{\hrulefill}\thinspace}

\providecommand{\MR}{\relax\ifhmode\unskip\space\fi MR }


\providecommand{\MRhref}[2]{
  \href{http://www.ams.org/mathscinet-getitem?mr=#1}{#2}

}

\providecommand{\href}[2]{#2}

\end{document}